\documentclass[12pt]{amsart}

\usepackage{pgf}
\usepackage{graphicx}
\usepackage{amsmath, amsthm, amssymb, amscd}

\newcommand{\eqe}{\approx_{\epsilon}}

\newcommand\Z{{\mathbb Z}}
\newcommand\N{{\mathbb N}}

\newcommand\bp{\begin{proof}}
\newcommand\ep{\end{proof}}

\newcommand\br{\begin{rem}}
\newcommand\er{\end{rem}}

\newcommand\bl{\begin{lem}}
\newcommand\el{\end{lem}}

\newcommand\bprop{\begin{prop}}
\newcommand\eprop{\end{prop}}

\newcommand\bt{\begin{thm}}
\newcommand\et{\end{thm}}

\newcommand\bc{\begin{cor}}
\newcommand\ec{\end{cor}}

\newcommand\ba{\begin{aligned}}
\newcommand\ea{\end{aligned}}

\newcommand\bi{\begin{itemize}}
\newcommand\ei{\end{itemize}}

\newcommand\bd{\begin{defn}}
\newcommand\ed{\end{defn}}

\newcommand\bq{\begin{q}}
\newcommand\eq{\end{q}}

\newcommand\cstar{$C^*$-algebra }

\newcommand\Qq{{\mathcal Q}}
\newcommand\Zz{{\mathcal Z}}

\DeclareMathOperator{\Ad}{Ad}
\DeclareMathOperator{\diag}{diag}

\DeclareMathOperator{\id}{id}

\theoremstyle{theorem}
\newtheorem{thm}{Theorem}[section]

\newtheorem*{thm*}{Theorem}

\newtheorem{prop}[thm]{Proposition}
\newtheorem{lem}[thm]{Lemma}
\newtheorem{cor}[thm]{Corollary}

\newtheorem{q}[thm]{Question}

\newtheorem*{conjzero*}{Conjecture Zero}
\newtheorem*{conjone*}{Conjecture One}
\theoremstyle{remark}
\newtheorem{rem}[thm]{Remark}
\newtheorem*{rem*}{Remark}

\theoremstyle{definition}
\newtheorem{defn}[thm]{Definition}

\address{\tiny Department of Mathematics, University of Oregon, Eugene, OR 97403-1222, USA}
\email{msun@uoregon.edu}
\author{Michael Sun}
\title{Existence of the Matui-Sato tracial Rokhlin property}
\begin{document}
\begin{abstract} We show by construction that when $G$ is an elementary amenable group and $A$ is a unital simple nuclear and tracially approximately divisible $C^*$-algebra, there exists an action $\omega$ of $G$ on $A$ with the tracial Rokhlin property in the sense of Matui and Sato. In particular, group actions with this Matui-Sato tracial Rokhlin property always exist for unital simple separable nuclear $C^*$-algebras with tracial rank at most one. If $A$ is simple with rational tracial rank at most one, then the crossed product $A\rtimes_{\omega}G$ is also simple with rational tracial rank at most one.
\end{abstract}
\maketitle

\section*{Introduction}

One fundamental way to investigate the structure of a $C^*$-algebra is through the study of its group actions.  An indispensable part of this theme is the crossed product construction. Given a group $G$, a $C^*$-algebra $A$ and a group action $\alpha$ of $G$ on $A$, we can construct a $C^*$-algebra $A\rtimes_{\alpha}G$ called the crossed product. When $G$ is discrete and $A$ is unital, we have the following presentation for the crossed product:
$$A\rtimes_{\alpha}G=\langle a, u_g\,|\,a\in A, g\in G, \alpha_g(a)=u_gau_g^*\rangle.$$
Despite its innocent appearance, it can be quite a challenge to determine what kind of $C^*$-algebra is being represented.

This brings us to another interesting aspect of the study of $C^*$-algebras. Namely the success of the classification of simple $C^*$-algebras using the \emph{Elliott invariant}. At the forefront of this success are the large classes of unital simple separable nuclear ($\Zz$-stable) $C^*$-algebras satisfying the Universal Coefficient Theorem and having \emph{tracial rank} zero, tracial rank at most 1, \emph{rational tracial rank} zero and rational tracial rank at most 1, which were discovered to be classifiable by Lin \cite{LinDuke}, \cite{LinCJM}, \cite{Lintad} and by Lin-Niu \cite{LinNiu} and Lin-Winter \cite{Winter}. Classifiable $C^*$-algebras necessarily possess a property called $\Zz$-stability which is attracting a lot of attention in work extending Elliott's classification program, while those $C^*$-algebras of tracial rank at most one are also known to belong to a larger class of algebras called \emph{tracially approximately divisible} $C^*$-algebras.

From the point of view of classification, the crossed product construction gives a way to explicitly construct algebras which belong to a large class of classifiable $C^*$-algebras that were otherwise only identifiable by their Elliott invariant.
Conversely, the classification of $C^*$-algebras also brings clarity to the crossed product construction itself by identifying algebras otherwise only defined by generators and relations.

Of major interest with respect to these two themes is a property of group actions called the \emph{tracial Rokhlin property}. Its appeal comes from the promise that algebras with desirable properties combined with actions with the tracial Rokhlin property can be used to construct crossed products with the same desirable properties and in many cases to even have the crossed product belong to the same class of classifiable algebras as the original algebra. Another advantage of great practical importance is that this property is believed to exist in abundance, which was lacking in its predecessor, the \emph{Rokhlin property}.

The question of whether this property exists will be the main focus of this article. Originally, the definition of the tracial Rokhlin property was only stated for $\Z$ (Osaka-Phillips \cite{OP}) and finite groups (Phillips \cite{P}), and there is an existence type result for $\Z$ in Lin \cite{Linkish} together with uniqueness. These build on the work of Kishimoto and others. The definition was then extended to include discrete amenable groups acting on certain $\Zz$-stable algebras in Matui-Sato \cite{MS2} where it was also shown that the crossed product retained the $\Zz$-stability of the original algebra. We state their definition more generally for $\Zz$-stable tracially approximately divisible $C^*$-algebras, which we will see is a natural setting for their definition. We will construct in this article, for every countable discrete group $G$ and $\Zz$-stable $C^*$-algebra $A$, an action $\omega$ of $G$ on $A$ to prove the following in Section 3 as Corollary \ref{onA}.  
\begin{thm*}
Given a countable discrete elementary amenable group $G$ and a unital simple separable $\Zz$-stable tracially approximately divisible $C^*$-algebra $A$,
then there exists a group action $\omega$ of $G$ on $A$ such that $\omega$ has the tracial Rokhlin property.
\end{thm*}

Built into the construction of $\omega$ is a family of $G$-actions $\gamma$ on the Jiang-Su algebra $\Zz$, which we introduce in Section 2.
A preliminary investigation is also undertaken to ascertain the classifiability of $\Zz\rtimes_{\gamma}G$ and $A\rtimes_{\omega}G$ in Sections 4 and 5. We show 
\begin{thm*}
Suppose $A$ is a unital simple $\Zz$-stable $C^*$-algebra with rational tracial rank at most one, suppose $G$ belongs to the class of groups generated by finite and abelian groups under increasing unions and taking subgroups and suppose $\omega$ is the action constructed to prove the above. Then $A\rtimes_{\omega} G$ is a unital simple $\Zz$-stable $C^*$-algebra with rational tracial rank at most one.
\end{thm*}
\begin{rem*}As one would have it, very recent results of Ozawa-R\o rdam-Sato \cite{ORS} imply that we can extend this to all elementary amenable groups. The difference in their result comes entirely from them being able to establish the Universal Coefficient Theorem for the crossed products they consider.
\end{rem*}
We can also replace rational tracial rank at most one by the classes of tracial rank zero or at most one and rational tracial rank zero and get the same result. This shows there is always an action present that has the tracial Rokhlin property and preserves the tracial rank.

\subsection*{Acknowledgements} I would like to thank Huaxin Lin for his advice and Yasuhiko Sato for showing me a proof of  \cite[Lemmas 6.12, 6.13]{MS2} and the referee for suggesting many helpful improvements. Most of this was done at the Research Center for Operator Algebras in Shanghai and can be found as a part of the author's PhD thesis.
\subsection*{Notation} 
We adopt the following conventions throughout this article:
\bi
\item Denote the cardinality of a set $S$ by $|S|$.\\
\item Write $a\approx_{\epsilon}b$ to stand for $\|a-b\|<\epsilon$.
\ei
\section{The tracial Rokhlin property for discrete groups}
	
We define the tracial Rokhlin property and strong outerness for discrete group actions on $C^*$-algebras with reference to Matui-Sato \cite{MS2}.

\bd Let $G$ be a countable discrete group, let $F\subset G$ be a finite subset and let $\epsilon>0$. We say a finite subset $K$ of $G$ is $(F,\epsilon)$-invariant if $K\neq\emptyset$ and
$$\left|K\cap\bigcap_{g\in F}g^{-1}K\right|\geq (1-\epsilon)|K|.$$
\ed
\bd A countable discrete group $G$ is said to be \emph{amenable} if an $(F,\epsilon)$-invariant subset exists from any $(F,\epsilon)$. The group $G$ is said to be \emph{elementary amenable} if it is contained in the smallest class of groups that contains all abelian groups, all finite groups and is closed under taking subgroups, quotients, direct limits and extensions.
\ed
The following definition of tracial Rokhlin property appears in Matui-Sato \cite{MS2} with the assumption of tracial rank zero and will be our working definition throughout this article with the assumption of tracial rank zero relaxed. When $A$ has strict comparison, this is equivalent to a natural generalization of the tracial Rokhlin property with the trace condition replaced by a Cuntz relation.

Let $T(A)$ denote the tracial state space of $A$.
\bd\label{mstrp} A group action $\alpha$ of $G$ on a \cstar $A$ has the \emph{Matui-Sato tracial Rokhlin property} if for every finite subset $F\subset G$ and $\epsilon>0$, there is a finite $(F,\epsilon)$-invariant subset $K$ in $G$ and a central sequence $(p_n)_{n\in\N}$ in $A$ consisting of projections such that 
\bi
\item $\lim_{n\to\infty}\alpha_g(p_n)\alpha_h(p_n)=0$ for all $g,h\in K$ such that $g\neq h$\\
\item $\lim_{n\to\infty}\max_{\tau\in T(A)}|\tau(p_n)-|K|^{-1}|=0$.
\ei
\ed
There is also the weaker notion of this for algebras without projections called the \emph{weak Rokhlin property} (Matui-Sato \cite[Definition 2.5]{MS2}).
Here is a condition which is a priori weaker than the properties above. 
\bd Let $A$ be a $C^*$-algebra and $\pi_{\tau}$ be the representation obtained from the Gelfand-Naimark-Segal construction with respect to the tracial state $\tau$. An automorphism is said to be \emph{weakly inner} if it is inner when considered as an automorphism of the weak closure $\pi_{\tau}(A)''$ for some trace $\tau$ invariant under the automorphism. An action $\alpha$ of $G$ on $A$ is \emph{strongly outer} if $\alpha_g$ is not weakly inner for all $g\in G\setminus\{1\}$.
\ed
The next lemma is taken from Matui-Sato \cite[Lemmas 6.12, 6.13]{MS2}. 
\begin{lem}\label{so}
Let $A$ be a simple $C^*$-algebra with unique trace $\tau$ and $\alpha$ an action of a group $G$ on $A$ such that $\alpha_g\neq\id$ when $g\neq1$.\\
Then $\alpha^{\otimes\N}=\bigotimes_{n\in\N}\alpha$ is a strongly outer action on $A^{\otimes\N}=\bigotimes_{n\in\N} A$.
\end{lem}
\begin{proof}Let $g\in G\setminus\{1\}$. Since the unitary group $U(A)$ spans $A$ and $\alpha_g\neq\id$, there is $v_g\in U(A)$ such that $\alpha_g(v_g)\neq v_g$. Now define a central sequence in $A^{\otimes\N}$ as follows:
$$v_g(n)=1\otimes\dots\otimes1\otimes v_g\otimes1\otimes\dots\otimes1\otimes\dots,$$
where $v_g$ appears in the $n$-th factor and $1$ appears in every other factor.
Let $\|a\|_2=\sup_{\tau\in T(B)}\tau(a^*a)^{1/2}$, which is a norm on $A$. We will show that if $\alpha^{\otimes\N}_g$ is weakly inner then 
$$\|\alpha^{\otimes\N}_g(v_{g}(n))-v_{g}(n)\|_2\to0,$$
which contradicts the sequence being constant and non-zero as follows:
$$\begin{aligned}\|\alpha^{\otimes\N}_g(v_g(n))-v_g(n)\|_{2}^2&=\tau^{\otimes\N}(1_{A^{\otimes(\N\setminus\{ n\})}}\otimes((\alpha_g (v_g)-v_g)^*(\alpha_g(v_g)-v_g)))\\
&=\tau((\alpha_g(v_g)-v_g)^*(\alpha_g(v_g)-v_g))\\
&=\|\alpha_g(v_g)-v_g\|_2^2.\end{aligned}$$
Now assume there is a unitary $u\in \pi_{\tau}(A)''$ such that $\alpha^{\otimes\N}_g=\Ad u$ on $\pi_{\tau}(A)''$. Then there is a sequence $(x_k)_{k\in\N}$ in $\pi_{\tau}(A)$ such that $x_k\to u$ strongly.  Such a $u$ allows us to extend $\pi_{\tau}$ to a representation of $A\rtimes_{\alpha_g}\Z$ and $\tau$ to a trace on $A\rtimes_{\alpha_g}\Z$ by
$$\tau_{A\rtimes\Z}(x)=\langle\pi_{A\rtimes\Z}(x)\tilde{1},\tilde{1}\rangle,$$
where $\tilde{1}$ is the cyclic vector that appears in the definition of the GNS-representation. Let $\epsilon>0$, fix $k$ so that $\|u-x_k\|_{2,A\rtimes\Z}\approx_{\epsilon/2}0$ by way of $x_k$ strongly converging to $u$, and let $n$ be large enough so that $[x_k,v_g(n)]\eqe0$, which is possible because $v_g(n)$ is a central sequence. We now calculate:
$$\begin{aligned}\|\alpha^{\otimes\N}_g(v_g(n))-v_g(n)\|_{2,A}&=\|uv_g(n)u^*-v_g(n)\|_{2,A\rtimes\Z}\\
											      &= \|uv_g(n)-v_g(n)u\|_{2,A\rtimes\Z}\\	
	 (\|ab\|_2\leq\|a\|_2\|b\|)\quad 				 &\leq 2\|u-x_k\|_{2,A\rtimes\Z}+\|x_kv_g(n)-v_g(n)x_k\|\\												  &\eqe\|x_kv_g(n)-v_g(n)x_k\|\\
			 &\eqe0.\end{aligned}$$\end{proof}
\begin{lem}\label{so2}
Let $A$ be a simple $C^*$-algebra with unique trace and $\alpha$ an action of a group $G$ on $A$ such that $\alpha_g\neq\id$ when $g\neq1$. Let $A_0$ be any simple $C^*$-algebra and $\alpha_0$ any action of $G$ on $A_0$, then $\alpha_0\otimes\alpha^{\otimes\N}$ is a strongly outer action on $A_0\otimes A^{\otimes\N}$.
\end{lem}
\bp
Since the argument used in the proof of Lemma \ref{so} only relies on the properties of the limit of $\|\alpha^{\otimes\N}_g(w_{g}(n))-w_{g}(n)\|_2$, we are free to change the first factor without affecting the conclusion.
\ep
\bt[Matui-Sato]\label{sotrp} Let A be a unital simple separable $C^*$-algebra with tracial rank zero and with finitely many extremal tracial states. Let $G$ be a countable discrete elementary amenable group. Then an action of $G$ on $A$ has the tracial Rokhlin property if and only if it is strongly outer.
\et
\bp This is Matui-Sato \cite[Theorem 3.7]{MS2} specialized to actions and to elementary amenable groups.
\ep
Note this works more generally for groups satisfying the so-called property (Q).
\section{A family of group actions on the Jiang-Su algebra}
We will introduce the Jiang-Su algebra $\Zz$ via its properties in the following proposition and only rely on these in later sections. 

\bprop\label{infinitez} The Jiang-Su algebra $\Zz$ is a unital simple separable nuclear $C^*$-algebra with a unique tracial state and an isomorphism 
$$\Zz\cong \varinjlim\left(\Zz^{\otimes n}, \id_{\Zz^{\otimes n}}\otimes1\right).$$
\eprop
\bp This appeared as \cite[Theorem 2.9, Corollary 8.8]{JiangSu}.
\ep
Another important example is the universal UHF-algebra $\Qq$.




\bd\label{gamma} Let $G$ be a countable discrete group and identify $\Zz$ with $\Zz^{\otimes G}=\bigotimes_{g\in G}\Zz$ using Lemma \ref{infinitez} and countability of $G$. Now define an action $\beta$ of $G$ on $\Zz$ via this identification by 
$$\beta_g: \bigotimes_{h\in G}z_{h}\mapsto\bigotimes_{h\in G}z_{g^{-1}h}.$$
Again using Lemma \ref{infinitez} to identify $\Zz$ with $\Zz^{\otimes\N}=\bigotimes_{n\in\N}\Zz$, we define the action $\gamma$ of $G$ on $\Zz$ by
$$\gamma_g: \bigotimes_{n\in\N}z_{n}\mapsto\bigotimes_{n\in\N}\beta_g(z_{n}).$$
For any group automorphism $\varphi$, we define $\beta^{\varphi}$ to be the action of $G$ on $\Zz$ given by $g\mapsto (\varphi\circ\beta)_g$. So we have that $\beta^{\id_G}=\beta$.
We can define $\gamma^{\varphi}$ analogously using $\beta^{\varphi}$ instead of $\beta$ and also get that $\gamma^{\id_G}=\gamma$.
\ed
We first show that all of the $\beta^{\varphi}$ are conjugate for different $\varphi$ so we only need to consider $\beta^{\id_G}=\beta$ from here without loss of generality. A similar thing happens when a different ordering of $G$ is taken to define the infinite tensor product.
\bprop Let $\varphi$ be a group automorphism of $G$ and let $\hat{\varphi}$ denote the induced automorphism on $\Zz$ given by
$$\hat{\varphi}: z=\bigotimes_{h\in G}z_{h}\mapsto\bigotimes_{h\in G}z_{\varphi(h)}.$$
Then for all $g\in G$, we have
$$\beta_g^{\varphi}=\hat{\varphi}\circ\beta_g\circ\hat{\varphi}^{-1}.$$
\eprop
\bp This is a straightforward calculation.
\ep
				
\bt\label{gammawrp}
Suppose $G$ is a countable discrete group. Then $\gamma$ is conjugate to $\beta$. In particular, both $\gamma$ and $\beta$ are strongly outer.
\et
\bp Since $\beta_g\neq\id$ for all $g\in G\setminus\{1\}$ and $\Zz$ has unique trace, we apply Lemma \ref{so} with $\alpha=\beta$ and $A=\Zz$ to get that $\gamma$ is strongly outer. Now we annotate with subscripts our copies of $\Zz$ in the tensor product decomposition of $\Zz^{\otimes G}$ to emphasise their position. That is,
$$\Zz^{\otimes G}=\bigotimes_{h\in G}\Zz_{h}.$$
The action $\beta$ acts by permuting these factors. We note that for each $h$, the factor $\Zz_h$ can be further decomposed using Lemma \ref{infinitez} into an infinite tensor product of copies of $\Zz$, where we denote each copy by $\Zz_h^{(l)}$ to emphasise its placement in the decomposition of $\Zz_h$. That is,
$$\Zz_{h}\cong\bigotimes_{l\in\N}\Zz_{h}^{(l)}.$$ 
We see that for each $l\in\N$, $\beta$ leaves the subalgebra $\Zz^{(l)}=\bigotimes_{h\in G}\Zz_h^{(l)}$ invariant and we recover the action $\beta$ when we identify $\Zz^{(l)}$ with $\Zz$. Hence we have that $\beta$ is conjugate to $\beta^{\otimes\N}$ acting on $\bigotimes_{l=1}^{\infty}\Zz^{(l)}$, which is conjugate to $\gamma$ acting on $\Zz$.
\ep
\section{Group actions on tracially approximately divisible $C^*$-algebras}
Recall that a $C^*$-algebra $A$ is called $\Zz$-stable if $A\otimes\Zz\cong A$.
\bd\label{omega} Let $\gamma$ be as in Definition \ref{gamma}. For any \cstar $A$, define the action $\omega^A$ on $A\otimes\Zz$ by
$$\omega^A=\id_A\otimes\gamma.$$
\ed
\begin{rem} If $A$ is unital, the action $\omega^A$ is pointwise approximately inner because all automorphisms on $\Zz$ are approximately inner. If $A$ is $\Zz$-stable, $\omega^A$ gives an action on $A$ which we will also call $\omega^A$.
\end{rem}
Let $M_k$ denote the full $k\times k$ matrix algebra with identity written $1_k$.  
\bl\label{ps} Suppose $G$ be an elementary amenable group. Then for any finite subset $F\subset G$ and $\epsilon>0$, there is a finite $(F,\epsilon)$-invariant subset $K$ of $G$ such that for each $n\in\N$ there is $m(n)\in\N$ and a projection $q_n\in M_{m(n)}\otimes\Zz$ satisfying the following
$$\ba
\omega_g^{M_{m(n)}}(q_n)\omega_h^{M_{m(n)}}(q_n)&\approx_{1/n}0\quad\text{for all $g,h\in K$ with $g\neq h$},\\
\tau(q_n)&\approx_{1/n}|K|^{-1}.\ea$$
\el
\bp Let $F$ be a finite subset of $G$ and let $\epsilon>0$. By Lemma \ref{so2}, $\id\otimes\gamma$ is a strongly outer action on $\Qq\otimes\Zz$. Therefore by Theorem \ref{sotrp} it also has the tracial Rokhlin property. So we have from Definition \ref{mstrp} a finite subset $K$ in $G$ and a central sequence $(q_n')$ consisting of projections in $\Qq\otimes\Zz$ such that for all $g,h\in K$ with $g\neq h$, we have
\bi
\item $\lim_{n\to\infty}(\id_{\Qq}\otimes\gamma_g)(q_n')(\id_{\Qq}\otimes\gamma_h)(q_n')=0$,\\
\item $\lim_{n\to\infty}\max_{\tau\in T(\Qq\otimes\Zz)}|\tau(q_n')-|K|^{-1}|=0$.
\ei
By passing to a subsequence if necessary and noting that $\Qq\otimes\Zz$ has unique trace, we have
$$\ba 
(\id_{\Qq}\otimes\gamma_g)(q_n')(\id_{\Qq}\otimes\gamma_h)(q_n')&\approx_{1/3n}0\quad\text{for all $g,h\in K$ with $g\neq h$},\\
\tau(q_n')&\approx_{1/n}|K|^{-1}.\ea$$
Since $\Qq$ is a UHF-algebra, there are $m(n)$ and $q_n''\in M_{m(n)}$ self-adjoint such that $q_n'\approx_{1/15n}q_n''$. Hence by functional calculus there is a projection $q_n\in M_{m(n)}$ such that $q_n\approx_{5/15n}q_n'$. We see now, since automorphisms are isometric and $\id_{\Qq}\otimes\gamma$ restricts to $\id_{m(n)}\otimes\gamma$ on $M_{m(n)}\otimes\Zz$, that
$$\ba 
(\id_{m(n)}\otimes\gamma_g)(q_n)(\id_{m(n)}\otimes\gamma_h)(q_n)&\approx_{1/3n}(\id_{\Qq}\otimes\gamma_g)(q_n')(\id_{m(n)}\otimes\gamma_h)(q_n)\\
						&\approx_{1/3n}(\id_{\Qq}\otimes\gamma_g)(q_n')(\id_{\Qq}\otimes\gamma_h)(q_n')\\
						&\approx_{1/3n}0\quad\text{for all $g,h\in K$ with $g\neq h$}.\ea$$
We also have
$$\tau(q_n)=\tau(q_n')\approx_{1/n}|K|^{-1}.$$
\ep

\bd\label{stad} Let $A$ be a unital simple separable $C^*$-algebra with strict comparison.
We say that $A$ is \emph{tracially approximately divisible} if for every $\epsilon > 0$, every $n\in\N$, every finite subset $\{a_1,a_2,\dots a_k\}\subset A$, there exists $N>n$ and a $*$-homomorphism $\varphi: M_{N}\to A$ such that for all $i\leq k$, $e\in M_{N}$ with $\|e\|\leq 1$ and $\tau\in T(A)$,\\
\bi
\item $[a_i,\varphi(e)]\eqe0$,\\
\item $\sup_{\tau\in T(A)}|1-\tau(\varphi(1_N))|\eqe0.$
\ei
\ed 
\begin{rem*}The definition above is equivalent to that with the third condition replaced by the standard Cuntz relation. In this case they are tracially $\Zz$-stable (see \cite[Definition 2.1, Theorem 3.3]{trzstable}) and will necessarily have strict comparison even if it is not explicitly assumed.
\end{rem*}

\bt\label{main} If $A$ is a unital simple separable tracially approximately divisible $C^*$-algebra, $G$ is a countable discrete elementary amenable group and $\omega^A$ is an action of $G$ on $A\otimes\Zz$ as in Definition \ref{omega}, then $\omega^A$ has the tracial Rokhlin property.
\et
\bp Let $F$ be a finite subset of $G$ and let $\epsilon>0$.
We aim to show that there is a finite $(F,\epsilon)$-invariant subset $K$ in $G$ and a central sequence $(p_n)_{n\in\N}$ consisting of projections in $A\otimes\Zz$ such that
\bi 
\item $\lim_{n\to\infty}\omega_g^A(p_n)\omega_h^A(p_n)=0$ for all $g,h\in K$,\\
\item $\lim_{n\to\infty}\max_{\tau\in T(A\otimes\Zz)}|\tau(p_n)-|K|^{-1}|=0$.
\ei
We begin by introducing some notation for this proof. Define
$$\Zz^{\otimes n}=\bigotimes_{j\leq n}\Zz\quad\text{with action $\beta^{\otimes n}=\otimes_{j\leq n}\beta$,}$$
and
$$\Zz^{\otimes(\N\setminus n)}=\bigotimes_{j> n}\Zz\quad\text{with action $\beta^{\otimes(\N\setminus n)}=\otimes_{j>n}\beta$}.$$
There are obvious action preserving isomorphisms
$$\rho_n: (\Zz,\gamma)\to (\Zz^{\otimes n}\otimes \Zz^{\otimes(\N\setminus n)},\beta^{\otimes n}\otimes\beta^{\otimes(\N\setminus n)})$$
and
$$\sigma_n:(\Zz,\gamma)\to (\Zz^{\otimes(\N\setminus n)},\beta^{\otimes(\N\setminus n)}).$$
Fix a dense sequence $x_1, x_2, \dots$ in $A\otimes\Zz$. We will proceed to define for each $n\in\N$ a projection $p_n$ to satisfy our initial requirements. To do this, it will be helpful to also establish for $j\leq n$,
$$[p_n,x_j]\approx_{5/n}0.$$
Let $n\in\N$. Find $a_{i,j}\in A$ and $z_{i,j}\in\Zz$ such that for $j\leq n$, we have
$$x_j\approx_{1/n}\sum_{i=1}^{l(j)}a_{i,j}\otimes z_{i,j}.$$
Write
$$L_n=\sum_{j\leq n}\sum_{i\leq l(j)}\|a_{i,j}\|.$$
There exists $n'\in\N$ such that for all $j\leq n$ and $i\leq l(j)$, there are $z_{i,j}'\in\Z^{\otimes n'}$ satisfying
$$\rho_{n'}(z_{i,j})\approx_{\frac{1}{nL_n}}z_{i,j}'\otimes1_{\Zz^{\otimes(\N\setminus n')}}.$$
Define an action preserving isomorphism
$$\chi_n: A\otimes \Zz\to A\otimes\Zz^{\otimes n'}\otimes\Zz$$
by 
$$\chi_{n}=(\id_{A\otimes\Zz^{\otimes n'}}\otimes\sigma_{n'}^{-1})\circ(\id_A\otimes\rho_{n'}).$$
For $j\leq n$ and $i\leq l(j)$, we get
\begin{equation}\label{chixj}\chi_{n}(x_j)\approx_{\frac{2}{n}}  \sum_{i=1}^{l(j)}a_{i,j}\otimes z_{i,j}'\otimes 1_{\Zz}.\end{equation}
(The calculation for (\ref{chixj}) is included here for convenience)
$$\ba \left\|\chi_{n}(x_j)-  \sum_{i=1}^{l(j)}a_{i,j}\otimes z_{i,j}'\otimes 1_{\Zz}\right\|&\approx_{1/n}\left\|\ \sum_{i=1}^{l(j)}a_{i,j}\otimes((\id\otimes\sigma_{n'}^{-1})\rho_{n'}(z_{i,j})-z_{i,j}'\otimes1_{\Zz})\right\|\\
														&=\left\|\ \sum_{i=1}^{l(j)}a_{i,j}\otimes((\id\otimes\sigma_{n'}^{-1})(\rho_{n'}(z_{i,j})-z_{i,j}'\otimes1_{\Zz^{\otimes(\N\setminus n')}}))\right\|\\
														&\leq \sum_{i=1}^{l(j)}\|a_{i,j}\|\|((\id\otimes\sigma_{n'}^{-1})(\rho_{n'}(z_{i,j})-z_{i,j}'\otimes1))\|\\
														&= \sum_{i=1}^{l(j)}\|a_{i,j}\|\|\rho_{n'}(z_{i,j})-z_{i,j}'\otimes1))\|\\
														&\leq \sum_{i=1}^{l(j)}\|a_{i,j}\|\frac{1}{nL_n}\\
														&=\frac{1}{nL_n}\sum_{i=1}^{l(j)}\|a_{i,j}\|\\
														&\leq\frac{L_n}{nL_n}\\
														&=\frac{1}{n}.\ea$$
We now apply Lemma \ref{ps} to get an $m\in\N$, a finite $(F,\epsilon)$-invariant subset $K$ of $G$ and a projection $q_n\in M_{m}\otimes\Zz$ satisfying:
$$\ba
\left((\id_{m}\otimes\gamma_g)(q_n)\right)\left((\id_{m}\otimes\gamma_h)(q_n)\right)&\approx_{1/n}0\quad\text{for all $g,h\in K$ with $g\neq h$},\\
(\tau_m\otimes\tau_{\Zz})(q_n)&\approx_{1/n}|K|^{-1}.\ea$$ 
Thinking of $q_n$ as a matrix with entries $y_{k,l}\in\Zz$, we have
$$q_n=\sum_{k,l=1}^{m}e_{k,l}\otimes y_{k,l},$$
where the $e_{k,l}$ are standard matrix units. Also define for convenience,
$$L_n'=\sum_{j=1}^n\sum_{i=1}^{l(j)}\sum_{k,l=1}^{m}\|y_{k,l}\|\|z_{i,j}'\|.$$     
Since $A$ is tracially approximately divisible, there exists by Definition \ref{stad}, an $m'\in\N$, a $*$-homomorphism $\varphi:M_{m'}\to A$,  satisfying for all  $j\leq n$, $i\leq l(j)$, and $e\in M_{m'}$ with $\|e\|\leq1$,
\begin{eqnarray}\label{ae}
[a_{i,j},\varphi(e)]&\approx_{\frac{1}{nL_n'}}&0\\
\label{tracem'}\tau(\varphi(1_{m'}))&\approx_{\frac{1}{n}}&1\\
\label{m'mn}m'&>&mn.\end{eqnarray} 
Now write $m'=Nm+r$ with $0\leq r<m$, and $N\in\N$ and define an embedding 
$$\psi_n:M_{m}\otimes\Zz\hookrightarrow A\otimes\Zz^{\otimes n'}\otimes\Zz$$
on generators for $e\in M_{m}$ and $z\in\Zz$ by
$$e\otimes z\mapsto \varphi(\diag(e\otimes1_N,0_r))\otimes1\otimes z,$$
where $\diag(e\otimes1_N,0_r)$ denotes a block diagonal matrix with the first $N$ blocks $e$ and zeros for the remaining $r\times r$ block. We see that this embedding respects the group action and the image of $q_n$ is
\begin{equation}\label{psinqn}\psi_n(q_n)=\sum_{k,l=1}^{m}\varphi(\diag(e_{k,l}\otimes1_N,0_r))\otimes1\otimes y_{k,l}.\end{equation}
We now define the promised projections $p_n\in A\otimes\Zz$ by
$$p_n=(\chi_n^{-1}\circ\psi_n)(q_n).$$
We first check for all $j\leq n$ that
$$[p_n,x_j]\approx_{5/n}0.$$
Let $j\leq n$. Then
$$\ba \chi_n([p_n,x_j])&=[\chi_n(p_n),\chi_n(x_j)]\\
				  &=[\psi_n(q_n),\chi_n(x_j)]\\
(\text{use (\ref{chixj})})\quad		  &\approx_{\frac{4}{n}}\left[\psi_n(q_n),\sum_{i=1}^{l(j)}a_{i,j}\otimes z_{i,j}'\otimes 1_{\Zz}\right]\\
(\text{use (\ref{psinqn})})\quad		  &=\sum_{i=1}^{l(j)}\sum_{k,l=1}^{m}[\varphi(\diag(e_{k,l}\otimes1_N,0_r))\otimes1\otimes y_{k,l},a_{i,j}\otimes z_{i,j}'\otimes1_{\Zz}]\\
		  &=\sum_{i=1}^{l(j)}\sum_{k,l=1}^{m}[\varphi(\diag(e_{k,l}\otimes1_N,0_r)),a_{i,j}]\otimes z_{i,j}'\otimes y_{k,l}\\	
(\text{use (\ref{ae})})\quad                &\approx_{1/n}0.\ea$$
We now show that $(p_n)_{n\in\N}$ satisfies the conditions in the definition of the tracial Rokhlin property. 
\bi
\item It is clear that $K$ is a finite $(F,\epsilon)$-invariant subset of $G$.\\
\item The sequence $(p_n)_{n\in\N}$ is central: 
Let $x\in A\otimes\Zz$ and $\epsilon>0$. We have from density that for some $j\in\N$,
$$x\approx_{\epsilon} x_j.$$
Now let $n\geq j$ be such that $1/n<\epsilon$, then
$$\ba      \chi_n([p_n,x]) &=[\chi_n(p_n),\chi_n(x)]\\
				&\approx_{2\epsilon}[\chi_n(p_n),\chi_n(x_j)]\\
				&\approx_{5/n}0.\ea$$
Hence for our choice of $n$ we have
$$[p_n,x]\approx_{7\epsilon}0.$$				
       
\item Orthogonality: Let $g,h\in K$ with $g\neq h$. Then
$$\ba \chi_n(\omega_g^A(p_n))&=\omega_g^A(\chi_n(p_n))\\
														&=(\id\otimes\beta_g^{\otimes n'}\otimes\gamma_g)(\psi_n(q_n))\\
														&=(\id\otimes\id\otimes\gamma_g)(\psi_n(q_n))\\
														&=\psi_n((\id\otimes\gamma_g)(q_n)).	\ea$$
So
$$\ba	\chi_n(\omega_g^A(p_n)\omega_h^A(p_n))&=\chi_n(\omega_g^A(p_n))\chi_n(\omega_h^A(p_n))\\
																							&=\psi_n((\id\otimes\gamma_g)(q_n))\psi_n((\id\otimes\gamma_h)(q_n))\\
																							&=\psi_n((\id\otimes\gamma_g)(q_n)(\id\otimes\gamma_h)(q_n))\\
																							&\approx_{1/n} 0.
	\ea$$									

\item Trace condition: let $\tau\in T(A)$, and let $\tau_{\Zz}$ and $\tau_k$ be the unique tracial states on $\Zz$ and $M_k$ respectively. Then
$$\ba
(\tau\otimes\tau_{\Zz})(p_n)&=(\tau\otimes\tau_{\Zz^{\otimes n'}}\otimes\tau_{\Zz})(\chi_n(p_n))\\
					&=(\tau\otimes\tau_{\Zz^{\otimes n'}}\otimes\tau_{\Zz})(\psi_n(q_n))\\
					&=(\tau\otimes\tau_{\Zz})(q_n)\\
					&=(\tau\otimes\tau_{\Zz})(\varphi(\diag(1_m\otimes1_N,0_r))\otimes1_{\Zz})(\tau_m\otimes\tau_{\Zz})(q_n)\\
					&=\tau(\varphi(\diag(1_m\otimes1_N,0_r)))(\tau_m\otimes\tau_{\Zz})(q_n)\\
					&=\tau(\varphi(1_{m'}))\tau_{m'}(\diag(1_m\otimes1_N,0_r))(\tau_m\otimes\tau_{\Zz})(q_n)\\
					&=\frac{m'-r}{m'}\tau(\varphi(1_{m'}))(\tau_m\otimes\tau_{\Zz})(q_n)\\
		(\text{use (\ref{tracem'})})\quad				&\approx_{1/n}\tau(\varphi(1_{m'}))(\tau_m\otimes\tau_{\Zz})(q_n)\\
			(\text{use (\ref{m'mn})})\quad			&\approx_{1/n}(\tau_m\otimes\tau_{\Zz})(q_n)\\
					&\approx_{1/n}|K|^{-1}.\ea$$
\ei
Therefore we have 
$$\max_{\tau\in T(A\otimes\Zz)}|\tau(p_n)-|K|^{-1}|\leq \frac{3}{n},$$
from which it follows the limit is $0$.
\ep
\bc\label{onA} For any unital simple separable $\Zz$-stable tracially approximately divisible $C^*$-algebra $A$ and any discrete countable elementary amenable group $G$, there exists a pointwise approximately inner action $\omega$ of $G$ on $A$ with the tracial Rokhlin property. Furthermore, $\omega$ can be taken to be isomorphic to $\omega^A$ from Theorem \ref{main}.
\ec
\bp If $A$ is $\Zz$-stable, $\omega^A$ in Theorem \ref{main} is such an action on $A$.
\ep
\bc\label{nucA} For any unital simple separable nuclear tracially approximately divisible $C^*$-algebra $A$ and any discrete countable elementary amenable group $G$, there exists a pointwise approximately inner action $\omega$ of $G$ on $A$ with the tracial Rokhlin property. Furthermore, $\omega$ can be taken to be isomorphic to $\omega^A$ from Theorem \ref{main}.
\ec
\bp Since simple tracially approximately divisible algebras are tracially $\Zz$-stable (\cite[Definition 2.1]{trzstable}), then $A$ being nuclear implies it is in fact $\Zz$-stable (see \cite[Theorem 4.1]{trzstable}). Now use Corollary \ref{onA}.
\ep
\bc If $A$ is a unital simple separable nuclear infinite dimensional $C^*$-algebra of tracial rank at most one and $G$ is any discrete countable elementary amenable group, then there exists a pointwise approximately inner action $\omega$ of $G$ on $A$ with the tracial Rokhlin property. Furthermore, $\omega$ can be taken to be isomorphic to $\omega^A$ from Theorem \ref{main}.
\ec
\bp Lin \cite[Theorem 5.4]{Lintad} shows that $A$ is tracially approximately divisible. Since $A$ is also nuclear, we can apply Corollary \ref{nucA}.
\ep
Note this works more generally for groups satisfying the so-called property (Q) in \cite{MS2}.
\section{The crossed products $\Zz\rtimes_{\gamma}G$}
Let $\gamma$ be as in Definition \ref{gamma}. We investigate the classifiability of $\Zz\rtimes_{\gamma}G$ by examining its rational tracial rank, that is, the tracial rank of $\Qq\otimes(\Zz\rtimes_{\gamma}G)$. Also recall from Theorem \ref{gammawrp} that $\Zz\rtimes_{\gamma}G\cong\Zz\rtimes_{\beta}G$.

We first summarize what can be deduced in a straightforward manner from the literature about $\Zz\rtimes_{\gamma}G$ in the following proposition.
\bprop\label{coqd} If $G$ is a countable discrete amenable, then $\Zz\rtimes_{\gamma}G$
\bi
	\item is unital and separable,\\
	\item simple,\\
	\item has a unique tracial state,\\
	\item is nuclear,\\
	\item and is $\Zz$-stable if $G$ is elementary amenable.
\ei
\eprop
\bp The crossed product is unital and separable since $\Zz$ is unital and separable, and $G$ is discrete and countable. We first give an argument here that $\Zz\rtimes_{\gamma}G$ has unique trace with reference to Kishimoto \cite{kishUO}. Now it suffices to show that $\Qq\otimes(\Zz\rtimes_{\gamma}G)$ has unique trace. But we have $\Qq\otimes(\Zz\rtimes_{\gamma}G)\cong(\Qq\otimes\Zz)\rtimes_{\id\otimes\gamma}G$ and $\id\otimes\gamma=\omega^{\Qq}$ is a strongly outer action (by Lemma \ref{so2}) on $\Qq\otimes\Zz$, which is isomorphic to $\Qq$. So \cite[Theorem 4.5]{kishUO} says that $\omega^{\Qq}$ is pointwise uniformly outer. Now the proof of \cite[Lemma 4.3]{kishUO} applied to each automorphism gives the result. For simplicity use Kishimoto \cite[Theorem 3.1]{kishO}. For nuclearity use Rosenberg \cite[Theorem 1]{JR}. For $\Zz$-stablity use Matui-Sato \cite[Corollary 4.11]{MS2}.
\ep

Recall that if $A\subset B(H)$ is separable then $A$ is a quasidiagonal set of operators if there exists an increasing sequence of finite rank projections, $q_1 \leq q_2 \leq q_3\dots$, such that for all $a\in A$, $[a,q_n]\to0$ and $q_n\to 1_H$ strongly. A separable $C^*$-algebra is quasidiagonal if it has a faithful representation whose image is a quasidiagonal set of operators.

\begin{rem}\label{subqd} It is clear from this definition that a subalgebra of a quasidiagonal $C^*$-algebra is quasidiagonal.
\end{rem}

\bc\label{qdrtr0cor} The following are equivalent:
\bi
\item $\Zz\rtimes_{\gamma}G$ is quasidiagonal\\
\item $\Zz\rtimes_{\gamma}G$ has rational tracial rank zero\\
\item $\Zz\rtimes_{\gamma}G$ has rational tracial rank at most one.
\ei
\ec
\bp
If $\Zz\rtimes_{\gamma}G$ is quasidiagonal, then Proposition \ref{coqd} combined with Matui-Sato \cite[Theorem 6.1]{MS3} allows us to conclude that $\Zz\rtimes_{\gamma}G$ has rational tracial rank zero. The next implication is obvious. Finally, if $\Zz\rtimes_{\gamma}G$ has rational tracial rank at most one, then it is isomorphic to a subalgebra of $\Qq\otimes(\Zz\rtimes_{\gamma}G)$. Since $\Qq\otimes(\Zz\rtimes_{\gamma}G)$ has tracial rank at most one, it is quasidiagonal by Lin \cite[Corollary 6.7]{Linqd}. Hence Remark \ref{subqd} tells us that $\Zz\rtimes_{\gamma}G$ is quasidiagonal.
\ep

To establish a condition for all elementary amenable groups it suffices by \cite{C} and \cite{O} to show it for abelian groups, finite groups and that the condition is preserved by increasing unions, extensions by finite groups and extensions by $\Z$.

\bprop\label{limitG} Suppose $G=\varinjlim (G_i,\varphi_i)$ with $\varphi_i$ injective. Then the $\Zz\rtimes_{\beta^i}G_i$ are quasidiagonal if and only if $\Zz\rtimes_{\beta}G$ is quasidiagonal.
\eprop
\bp It suffices to show that $\Zz\rtimes_{\beta}G$ is an injective inductive limit of the $\Zz\rtimes_{\beta^i}G_i$. Use the injectivity of $\varphi_i$ to define a $*$-homomorphism 
$$\Psi_i: \Zz^{\otimes G_i}\to\Zz^{\otimes G_{i+1}},$$ 
on generators by
$$z=\bigotimes_{g\in G_i}z_g\mapsto 1_{\Zz^{\otimes(G_{i+1}\setminus\varphi_i(G_i))}}\otimes\bigotimes_{h\in\varphi_i(G_i)}z_{\varphi_i^{-1}(h)}.$$
We check this is covariant with respect to $\beta^i$. For $g\in G_i$, we have
$$\ba
\varphi_i(g)\Psi_i(z)\varphi_i(g)^*&=\beta_{\varphi_i(g)}^{i+1}\left(1_{\Zz^{\otimes(G_{i+1}\setminus\varphi_i(G_i))}}\otimes\bigotimes_{h\in\varphi_i(G_i)}z_{\varphi_i^{-1}(h)}\right)\\
&=1_{\Zz^{\otimes(G_{i+1}\setminus\varphi_i(G_i))}}\otimes\bigotimes_{h\in\varphi_i(G_i)}z_{\varphi_i^{-1}(\varphi_i(g)^{-1}h)}\\
&=1_{\Zz^{\otimes(G_{i+1}\setminus\varphi_i(G_i))}}\otimes\bigotimes_{h\in\varphi_i(G_i)}z_{g^{-1}\varphi_i^{-1}(h)}\\
&=\Psi_i\left(\bigotimes_{h\in G_i}z_{g^{-1}h}\right)\\
&=\Psi_i(\beta_g^i(z)).\ea$$
So we have a sequence of injective maps $\Psi_i\rtimes\varphi_i:\Zz\rtimes_{\beta^i}G_i\to\Zz\rtimes_{\beta^{i+1}}G_{i+1}$, whose limit we now show is isomorphic to $\Zz^{\otimes G}\rtimes_{\beta}G$. First notice that $\varinjlim(\Zz^{\otimes G_i},\Psi_i)\cong\Zz^{\otimes G}$ and get the obvious maps
$$\Zz^{\otimes G_i}\hookrightarrow\Zz^{\otimes G}\hookrightarrow\Zz^{\otimes G}\rtimes_{\beta}G$$
and
$$G_i\hookrightarrow G\hookrightarrow \Zz^{\otimes G}\rtimes_{\beta}G.$$
We can show covariance in much the same way as before to get
$$\Zz^{\otimes G_i}\rtimes_{\beta^i}G_i\hookrightarrow \Zz^{\otimes G}\rtimes_{\beta}G.$$
We check that these maps are compatible with the increasing $i$ and conclude there is an injective map
$$\varinjlim(\Zz^{\otimes G_i}\rtimes_{\beta^i}G_i,\Psi_i\rtimes\varphi_i)\hookrightarrow \Zz^{\otimes G}\rtimes_{\beta}G.$$
This map is also surjective because its image contains $\Zz^{\otimes G}$ and $G$, which generate the target algebra.
\ep

\bprop\label{gfinrtr0} If $G$ is a finite group, then $\Zz\rtimes_{\gamma}G$ is quasidiagonal.
\eprop
\bp Let $n=|G|$. We define an embedding of $\Zz\rtimes_{\gamma} G$ into $M_n(\Zz)$ as follows. First define $A\to M_n(\Zz)$ for $a\in A$ by
$$a\mapsto\diag (\gamma_g(a))_{g\in G}.$$
Then define $G\to U(M_n)\otimes1_{\Zz}\to U(M_n(\Zz))$ via its left regular representation.
One check that these satisfy the covariance relations and hence we get our embedding after noting $\Zz\rtimes_{\gamma} G$ is simple.
\ep
\begin{rem}\label{fext} This can be generalized using an elementary argument to show that the class of $G$ for which $\Zz\rtimes_{\gamma}G$ is quasidiagonal is preserved by finite extensions (see for example \cite[Proposition 2.1 (iv)]{ORS}).
\end{rem}
\bprop\label{gabqd} If $G$ is a countable discrete abelian group, then $\Zz\rtimes_{\gamma}G$ is quasidiagonal.
\eprop
\bp Lin \cite[Theorem 9.11]{fgLin} shows that if the crossed product of any AH-algebra and any finitely generated abelian group has an invariant tracial state, then the crossed product is quasidiagonal. We apply this to $(\Qq\otimes\Zz)\rtimes_{\id\otimes\gamma}G\cong\Qq\otimes(\Zz\rtimes_{\gamma}G)$ when $G$ is a finitely generated abelian group to show the crossed product is quasidiagonal. Now $\Zz\rtimes_{\gamma}G$ is a subalgebra of $\Qq\otimes(\Zz\rtimes_{\gamma}G)$ and hence quasidiagonal. The condition on $G$ being finitely generated can be removed by Proposition \ref{limitG}. 
\ep
\begin{rem}\label{Zext} If we examine this proof, we see that it will show that the class of $G$ for which $\Zz\rtimes_{\gamma}G$ is quasidiagonal is preserved by $\Z$-extensions provided $\Zz\rtimes_{\gamma}G$ satisfies the UCT. This was done for $M_{2^{\infty}}\rtimes_{\beta}G$ in \cite{ORS} by noticing that it is a groupoid and applying results of Tu \cite{Tu}. So $M_{2^{\infty}}\rtimes_{\beta}G$ is quasidiagonal, which implies $\Zz\rtimes_{\gamma}G\cong\Zz\rtimes_{\beta}G\subset M_{2^{\infty}}\rtimes_{\beta}G$ is quasidiagonal.
\end{rem}
We now summarize this section.
\bt\label{Ccrtr0} Let $\gamma$ be as in Definition \ref{gamma} and let $G$ be a countable discrete elementary amenable group. Then $\Zz\rtimes_{\gamma}G$ is a unital separable simple nuclear $\Zz$-stable $C^*$-algebra with rational tracial rank zero.
\et
\bp We combine Propositions \ref{limitG}, \ref{gfinrtr0} and \ref{gabqd} along with Remarks \ref{fext} and \ref{Zext} to get $\Zz\rtimes_{\gamma}G$ is quasidiagonal for all elementary amenable $G$ and hence has rational tracial rank zero by Corollary \ref{qdrtr0cor}. The remaining properties are those listed in Proposition \ref{coqd}.
\ep
\bc\label{KZzZ} Let $\gamma$ be as in Theorem \ref{Ccrtr0}. Then $\Zz\rtimes_{\gamma}\Z$ is unital separable simple nuclear $\Zz$-stable with rational tracial rank zero and a unique tracial state as well as satisfying the UCT. Also for $i=0$ or $1$,
$$K_i(\Zz\rtimes_{\gamma}\Z)=\Z.$$
\ec
\bp Putting $G=\Z$ in Theorem \ref{Ccrtr0} shows that $\Zz\rtimes_{\gamma}\Z$ is unital separable simple nuclear $\Zz$-stable with unique tracial state and rational tracial rank zero. Crossed products by $\Z$ always satisfy the UCT. The $K$-groups are obtained using the Pimsner-Voiculescu six-term exact sequence. 
\ep
We also know that every other strongly outer $\Z$-action is outer conjugate to $\gamma$ by Sato \cite[Theorem 1.3]{SatoZ}. 

\section{The crossed products $A\rtimes_{\omega}G$}
Let $A$ be a unital $C^*$-algebra and $G$ a discrete group. Let $\omega$ be as in Definition \ref{omega} and $\gamma$ as in Definition \ref{gamma}. If $A$ is $\Zz$-stable, then $\omega$ is conjugate to an action of $G$ on $A$ that we will also call $\omega$. 
\bprop\label{factor} For $g\in G$ let $u_g$ and $u_g'$ be the implementing unitaries for $\omega_g$ and $\gamma_g$ respectively. There is an isomorphism 
$$i:(A\otimes\Zz)\rtimes_{\omega}G\to A\otimes(\Zz\rtimes_{\gamma}G),$$
such that 
$$i: (a\otimes z)u_g\mapsto a\otimes(zu_g').$$
\eprop
\bp
This is standard.
\ep
\bc  $\Zz\rtimes_{\gamma}G\cong\Zz\rtimes_{\omega}G.$\qed
\ec


\bl\label{qfactor} If $A$ is $\Zz$-stable, then there is a $*$-isomorphism
$$\Qq\otimes(A\rtimes_{\omega} G)\cong(\Qq\otimes A)\otimes((\Zz\rtimes_{\gamma}G)\otimes\Qq).$$
\el
\bp We use $\Qq\cong\Qq\otimes\Qq$ and Proposition \ref{factor} to write
$$\ba (\Qq\otimes A)\otimes((\Zz\rtimes_{\gamma}G)\otimes\Qq)&\cong \Qq\otimes (A\otimes(\Zz\rtimes_{\gamma}G))\\
								&\cong \Qq\otimes((A\otimes\Zz)\rtimes_{\omega} G)\\
								&\cong \Qq\otimes (A\rtimes_{\omega}G).\ea$$
Now since $A$ is $\Zz$-stable, we are done.
\ep
\bprop\label{ussn} Suppose $A$ is a unital separable $\Zz$-stable $C^*$-algebra and $G$ is any countable discrete amenable group. Then $A\rtimes_{\omega}G$ is a unital separable $\Zz$-stable $C^*$-algebra and we also have:
\bi 
\item If $A$ is simple, then $A\rtimes_{\omega}G$ is simple.\\
\item If $A$ is nuclear, then $A\rtimes_{\omega}G$ is nuclear.\\
\item $T(A\rtimes_{\omega}G)=T(A)$.\\
\item If $A$ has real rank zero, then $A\rtimes_{\omega}G$ has real rank zero.
\ei
\eprop
\bp For $\Zz$-stablity we use Proposition \ref{factor} to get
$$\ba \Zz\otimes(A\rtimes_{\omega}G)&\cong \Zz\otimes (A\otimes(\Zz\rtimes_{\gamma}G))\\
																		&\cong (\Zz\otimes A)\otimes (\Zz\rtimes_{\gamma}G).\\
																		&\cong A\otimes(\Zz\rtimes_{\gamma}G)\\
																		&\cong A\rtimes_{\omega}G.\ea$$
Propositions \ref{coqd} and \ref{factor} imply all but the last two claims. For the tracial state spaces, let $\tau_{\gamma}$ be the unique tracial state on $\Zz\rtimes_{\gamma}G$ (Proposition \ref{coqd}) and define a map $T(A)\to T(A\otimes(\Zz\rtimes_{\gamma}G))$ given by $\tau\mapsto \tau\otimes\tau_{\gamma}$. This map is obviously affine and injective, while for surjectivity we make use of a brief argument which can be found as \cite[Lemma 5.15]{Linsome}. Now by Proposition \ref{factor} we have that $T(A\otimes(\Zz\rtimes_{\gamma}G))\cong T(A\rtimes_{\omega}G)$.  Since our algebras are $\Zz$-stable, we use the characterization of real rank zero by R\o rdam \cite[Theorem 7.2]{Rordamrrz} that $K_0(A)$ is uniformly dense in the space of affine functions on $T(A)$ under the standard mapping $\rho_A$ gotten by evaluation. Since $K_0(A)\subset K_0(A\rtimes_{\omega} G)$ via $p\mapsto p\otimes 1$, under the identification $T(A)=T(A\rtimes_{\omega}G)$, $\rho_{A\rtimes_{\omega}G}(K_0(A))=\rho_A(K_0(A))$, which is already uniformly dense. Hence the image of $\rho_{A\rtimes_{\omega}G}$ is uniformly dense and we are done. 
\ep
\bt\label{rtr1} Suppose $A$ is a unital separable simple nuclear $\Zz$-stable $C^*$-algebra and $G$ is a countable discrete elementary amenable group. Let $\omega$ be as in Definition \ref{omega} and let $\gamma$ be as in Definition \ref{gamma}. Then $A\rtimes_{\omega}G$ is a unital separable simple nuclear $\Zz$-stable $C^*$-algebra and:
\bi
\item If $A$ has rational tracial rank at most one, then $A\rtimes_{\omega}G$ has rational tracial rank at most one.\\
\item If $A$ has rational tracial rank zero, then $A\rtimes_{\omega}G$ has rational tracial rank zero.\\
\item If $A$ has tracial rank at most one, satisfies the UCT and $\Zz\rtimes_{\gamma}G$ satisfies the UCT, then $A\rtimes_{\omega}G$ has tracial rank at most one and satisfies the UCT.\\
\item If $A$ has tracial rank zero, satisfies the UCT and $\Zz\rtimes_{\gamma}G$ satisfies the UCT, then $A\rtimes_{\omega}G$ has tracial rank zero and satisfies the UCT.
\ei
\et
\bp By Proposition \ref{ussn}, the conditions of being unital, separable, simple, nuclear and $\Zz$-stable, are retained by the crossed product. To determine the rational tracial rank of $A\rtimes_{\omega}G$ we use Lemma \ref{qfactor} and apply Hu-Lin-Xue \cite[Theorem 4.8]{LinHuXue}, which says that the tracial rank of a tensor product is bounded by the sum of the tracial ranks of the factors, to the algebra on the right hand side of the lemma. The tracial rank of $\Qq\otimes(\Zz\rtimes_{\gamma}G)$ is zero by Theorem \ref{Ccrtr0}, which means the tracial rank is bounded by the rational tracial rank of $A$. This gives us both claims about rational tracial rank. Now we address the claim for $A$ being of tracial rank at most one. We will use \cite[Theorem 4.7]{LinWei} with our $A$ as their $B$ and $\Zz\rtimes_{\gamma}G$ as their $A$ to show that $A\otimes (\Zz\rtimes_{\gamma}G)$ has tracial rank at most one and satisfies the UCT. Hence by Proposition \ref{factor} $A\rtimes_{\omega}G$ has tracial rank at most one and satisfies the UCT.  Now if $A$ was also tracial rank zero, then it is real rank zero and Proposition \ref{ussn} tells us that $A\rtimes_{\omega}G$ is real rank zero. But we know that if an algebra is unital simple of tracial rank at most one and real rank zero, then it has tracial rank zero (see for example \cite[Lemma 3.2]{LinRok}).
\ep

Here is a curious result in the converse direction.
\bt\label{conversecor} Let $A$ be a unital separable simple $\Zz$-stable $C^*$-algebra and let $G$ be a countable discrete elementary amenable group. Let $\omega$ be as in Definition \ref{omega} and let $\gamma$ be as in Definition \ref{gamma}. If $\Zz\rtimes_{\gamma}G$ satisfies the UCT and $A\rtimes_{\omega}G$ has rational tracial rank at most one, then $A$ has rational tracial rank at most one.
\et
\bp Let $B=\Zz\rtimes_{\gamma}G$ and note $A\otimes B=A\rtimes_{\omega}G$ by Proposition \ref{factor}. We now apply Lin-Sun \cite[Theorem 4.2]{LinWei} using Propositions \ref{coqd} and \ref{ussn}, and Theorem \ref{Ccrtr0} to verify their hypotheses. 
\ep
We specialise Theorems \ref{rtr1} and \ref{conversecor} to the case of the integers.
\bc\label{AZ} Suppose $A$ is a unital separable simple nuclear $\Zz$-stable $C^*$-algebra satisfying the UCT and $\omega$ is as in Theorem \ref{main}. Then $A\rtimes_{\omega}\Z$ is a unital separable simple nuclear $\Zz$-stable $C^*$-algebra satisfying the UCT. We also have
\bi
\item $A$ has rational tracial rank at most one if and only if $A\rtimes_{\omega}\Z$ has rational tracial rank at most one. \\
\item If $A$ has rational tracial rank zero, then $A\rtimes_{\omega}\Z$ has rational tracial rank zero.\\
\item If $A$ has tracial rank at most one, then $A\rtimes_{\omega}\Z$ has tracial rank at most one.\\
\item If $A$ has tracial rank zero, then $A\rtimes_{\omega}\Z$ has tracial rank zero. \\
\item $K_i(A\rtimes_{\omega}\Z)=K_0(A)\oplus K_1(A)$ for $i=0$ or $1$.
\ei
\ec
\bp Proposition \ref{ussn} tells us that $A\rtimes_{\omega}\Z$ is unital separable simple nuclear and $\Zz$-stable. The UCT is always preserved by crossed products by $\Z$. For the claim about rational tracial rank, the forward implication is given by Theorem \ref{rtr1} with $G=\Z$ and the converse by Theorem \ref{conversecor} with $G=\Z$.  Since $\Zz\rtimes_{\gamma}\Z$ always satisfies the UCT, the next three claims all follow from Theorem \ref{rtr1} with $G=\Z$. For the $K$-groups, we use the Kunneth Theorem for tensor products combined with knowing the $K$-groups of $\Zz\rtimes_{\gamma}\Z$ from Corollary \ref{KZzZ}.
\ep 

\end{document}